\documentclass[12pt]{article}
\usepackage[a4paper,width=6.8in,height=8.8in]{geometry}
\usepackage{amsmath,amssymb}
\usepackage[hidelinks]{hyperref}
\usepackage{setspace}
\usepackage[utf8]{inputenc}
\usepackage[english]{babel}
\usepackage{booktabs,multirow}
\usepackage{blkarray}
\usepackage{lscape}
\usepackage[T1]{fontenc}
\usepackage{authblk,color}
\usepackage[english]{babel}
\usepackage{amsfonts, amstext, amsthm, booktabs, dcolumn}
\usepackage{graphicx}
\usepackage{float}
\usepackage{caption}
\usepackage{centernot}
\usepackage{subcaption}
\usepackage{times}
\usepackage[nottoc]{tocbibind}
\usepackage[numbers]{natbib}
\usepackage{mathrsfs}

\definecolor{gr}{rgb}{0.7, 0.0, 0.49}

\newtheorem{proposition}{\bf Proposition}[section]
\newtheorem{theorem}{\bf Theorem}[section]

\newtheorem{corollary}{\bf Corollary}[section]
\newtheorem{remark}{\bf Remark}[section]
\newtheorem{remarks}{\bf Remarks}[section]
\newtheorem{lemma}{\bf Lemma}[section]

\newtheorem{example}{\bf Example}[section]

\newcommand\floor[1]{\left\lfloor#1\right\rfloor}

\numberwithin{equation}{section}
\allowdisplaybreaks
\title{Approximations Related to the Sums of $m$-dependent Random Variables}
\author[1]{Amit N. Kumar}
\author[2]{Neelesh S. Upadhye}
\author[3]{P. Vellaisamy}
\affil[1,3]{\small Department of Mathematics, Indian Institute of Technology Bombay,}
\affil[ ]{\small Powai, Mumbai-400076, India.}
\affil[2]{\small Department of Mathematics, Indian Institute of Technology Madras,}
\affil[ ]{\small Chennai-600036, India.}
\affil[1]{\small Email: amit.kumar2703@gmail.com}
\affil[2]{\small Email: neelesh@iitm.ac.in}
\affil[3]{\small Email: pv@math.iitb.ac.in}
\date{}

\begin{document}
\maketitle

\begin{abstract}
\noindent
In this paper, we consider the sums of non-negative integer valued $m$-dependent random variables, and its approximation to the power series distribution. We first discuss some relevant results for power series distribution such as Stein operator, uniform and non-uniform bounds on the solution of Stein equation, and etc. Using Stein's method, we obtain the error bounds for the approximation problem considered. As special cases, we discuss two applications, namely, $2$-runs and $(k_1,k_2)$-runs and compare the bound with the existing bounds.
\end{abstract}

\noindent

\section{Introduction and Preliminaries}
The sums of $m$-dependent random variables (rvs) has special attention due to its applicability in many real-life applications such as runs and patterns, DNA sequences, and reliability theory, among many others. However, its distribution is difficult or sometimes intractable, especially if the setup is arising from non-identical rvs concentrated on $\mathbb{Z}_+=\{0,1,2,\ldots\}$, the set of non-negative integers. Therefore, there is a need to approximate such a distribution with some known and easy-to-use distributions. In this article, we consider power series distribution (PSD) approximation to the sums of $m$-dependent rvs. Approximations related to the sums of locally dependent rvs have been studied by several authors such as Barbour and Xia \cite{BX,BX2001}, Fu and Johnson \cite{FJ2009}, Vellaisamy \cite{V}, Wang and Xia \cite{WX}, and Soon \cite{soon}, among many others.\\
A sequence of rvs $\{X_k\}_{k\ge 1}$ is called $m$-dependent if $\sigma(X_1,X_2,\dotsc,X_i)$ and $\sigma(X_j,X_{j+1},\dotsc)$, for $j-i> m$, are independent, where $\sigma(X)$ denotes the sigma-algebra generated by $X$. The sums of $m$-dependent rvs can be reduced to the sums of $1$-dependent rvs, using rearrangement of rvs (see Section \ref{7:AR} for details). We mainly focus on the sums of $1$-dependent rvs concentrated on $\mathbb{Z}_+$, and obtain the error bounds. Of course,  the bound can directly apply for special distributions of PSD family. An advantage of approximation to PSD family is that we can obtain the error bounds for approximation to some specific distributions such as Poisson and negative binomial distributions. For some related works, we refer the reader to Lin and Liu \cite{LL2012}, \v{C}ekanavi\v{c}ius and Vellaisamy \cite{CV2015}, and references therein.\\
For $\mathbb{Z}_+$-valued rvs $X$ and $X^*$, the total variation distance is given by
\begin{align}
d_{TV}(X,X^*)=\frac{1}{2}\sum_{k=0}^{\infty}|\mathbb{P}(X=k)-\mathbb{P}(X^*=k)|.\label{dtv1}
\end{align} 
Hereafter, ${\bf 1}_A$ denotes the indicator function of $A\subseteq \mathbb{Z}_+$. Let $X$ be a rv concentrated on $\mathbb{Z}_+$, 
\begin{align*}
\mathcal{G}=\{f:\mathbb{Z}_+\to\mathbb{R}~|~\text{$f$ is bounded}\}
\end{align*} 
and
\begin{equation}
\mathcal{G}_X=\{g \in \mathcal{G}~|~g(0)=0~\text{and}~g(x)=0~\text{for}~x \notin \text{supp}(X)\},\label{7:g}
\end{equation}
associated with the rv $X$, where supp($X$) denotes the support of the rv $X$. We now briefly discuss Stein's method (Stein \cite{stein}) which we use to derive our approximation results in Section \ref{7:AR}. The Stein's method can be discussed in following three steps.
\begin{enumerate}
\item[(a)] Identify a Stein operator, denoted by $\mathcal{A}_X$ for a rv $X$, such that $\mathbb{E}[\mathcal{A}_Xg(X)]=0$, for $g\in \mathcal{G}_X$.
\item[(b)] Solve the Stein equation ${\cal A}_Xg(k)=f(k)-{\mathbb E}f(X)$, for $f \in {\cal G}$ and $g \in {\cal G}_X$.
\item[(c)] Replace $k$ by a rv $Y$ in Stein equation, and taking expectation and supremum to get
\begin{equation*}
d_{TV}(X,Y):=\sup_{f \in {\cal H}}|{\mathbb E}f(X)-{\mathbb E}f(Y)|=\sup_{f \in {\cal H}}|{\mathbb E}{\cal A}_Xg_f(Y)|,
\end{equation*}
where $g_f$ is the solution of the Stein equation and ${\cal H}=\{{\bf 1}_A| A \subseteq {\mathbb Z}_+\}$.
\end{enumerate} 
For additional details on Stein's method, see Barbour {\em et al.} \cite{BHJ}, Barbour and Chen \cite{BC1}, Ley {\em et al.} \cite{LRS}, Reinert \cite{R1}, Upadhye {\em et al.} \cite{UCV}, and the references therein.

\noindent
This article is organized as follows. In Section \ref{7:psdRR}, we discuss the PSD and its related results to Stein's method. In Section \ref{7:AR}, we derive the error bound for PSD approximation to the sums of $1$-dependent rvs and discuss some relevant remarks. In Section \ref{7:APTK1K2R}, we discuss two important applications of our results to the sums of $2$-runs and $(k_1,k_2)$-runs.

\section{Power Series Distribution and Related Results}\label{7:psdRR}
Let $Z$ be a $\mathbb{Z}_+$-values rv. We say its distribution belongs to the PSD family, denoted by $\mathcal{P}$, if $\mathbb{P}(Z=k)=p_k$ is of the form 
\begin{equation}
p_k=\frac{a_k \theta^k}{\gamma(\theta)}, \quad k\in\mathbb{Z}_+,\label{7:psdpmf}
\end{equation}
where $\theta>0$ and $a_k$, $k \ge 0$, are called series parameter and coefficient function, respectively. Many distributions such as Poisson, binomial, negative binomial, logarithmic series, and inverse sine distributions, among many others, belong to the PSD family. For more details, we refer the reader to Edwin \cite{ETK2014}, Noack \cite{noack}, Patil \cite{patil}, and the references therein. \\
Next, we give a brief discussion about Stein's method for PSD, in fact, many results follow from Eichelsbacher and Reinert \cite{ER}. The following proposition gives a Stein operator for PSD.

\begin{proposition}
Let the rv $Z$ having distribution belonging to PSD family defined in \eqref{7:psdpmf}. Then a Stein operator for $Z$ is given by
\begin{equation}
{\cal A}_Zg(k)=\theta (k+1)\frac{a_{k+1}}{a_k}g(k+1)-k g(k),\quad g\in\mathcal{G}_Z,~k \in \mathbb{Z}_+.\label{7:sopsd}
\end{equation}
\end{proposition}

\begin{proof} From \eqref{7:psdpmf}, it can be easily verified that
\begin{equation}
\theta (k+1)\frac{a_{k+1}}{a_k}p_k-(k+1)p_{k+1}=0.\label{7:dipmf}
\end{equation}
Let $g \in \mathcal{G}_Z$ defined in \eqref{7:g}, then 
$$\sum_{k=0}^{\infty}g(k+1)\left[\theta (k+1) \frac{a_{k+1}}{a_k}p_k-(k+1)p_{k+1}\right]=0.$$
Rearranging the terms, we have
$$\sum_{k=0}^{\infty}\left[\theta (k+1)\frac{a_{k+1}}{a_k}g(k+1)-k g(k)\right]p_k=0.$$
This proves the result.
\end{proof}

\noindent
Now, we discuss the solution of the Stein equation
\begin{align}
\theta (k+1)\frac{a_{k+1}}{a_k}g(k+1)-k g(k)=f(k)-\mathbb{E}f(Z),\quad f \in \mathcal{G},~g\in\mathcal{G}_Z.\label{7:se}
\end{align} 
Next we describe discrete Gibbs measure (DGM), a large class of distributions, studied by Eichelsbacher and Reinert \cite{ER}. If a rv $U$ has the distribution of the form
\begin{align}
\mathbb{P}(U=k)=\frac{1}{{\cal C}_w}e^{V(k)}\frac{w^k}{k!},\quad k \in \mathbb{Z}_+,\label{7:DGM}
\end{align}
for some function $V:\mathbb{Z}_+\to \mathbb{R}$, $w>0$, and ${\cal C}_w=\sum_{k=0}^{\infty}e^{V(k)}\frac{w^k}{k!}$, then we say the rv $U$ belongs to the DGM family. Observe here the support is $\mathbb{Z}_+$. Note that if we take $a_k=e^{V(k)}/k!\iff V(k)=\ln(a_k k!)$, $\theta=w$, and $\gamma(\theta)={\cal C}_w$, which are valid choices, then the results derived by Eichelsbacher and Reinert \cite{ER} are valid for PSD family. Therefore, the solution of \eqref{7:se} can be directly obtained from $(2.5)$ and $(2.6)$ of Eichelsbacher and Reinert \cite{ER} and is given by
\begin{align*}
g(k)&=\frac{1}{k a_k \theta^k}\sum_{j=0}^{k-1}a_j \theta^j[f(j)-\mathbb{E}f(Z)]\\
&=-\frac{1}{k a_k \theta^k}\sum_{j=k}^{\infty}a_j \theta^j[f(j)-\mathbb{E}f(Z)].
\end{align*}
Also, the Lemma $2.1$ of Eichelsbacher and Reinert \cite{ER} can be written for PSD family in the following manner.

\begin{lemma}
Let ${\cal G}_1=\{f:\mathbb{Z}_+\to[0,1]\}$, $F(k)=\sum_{i=0}^{k}p_i$ and $\bar{F}(k)=\sum_{i=k}^{\infty}p_i$. Assume that
\begin{align*}
k \frac{F(k)}{F(k-1)}\ge \theta (k+1)\frac{a_{k+1}}{a_k}\ge k \frac{\bar{F}(k+1)}{\bar{F}(k)}.
\end{align*}
Then, for $f \in \mathcal{G}_1$ and $g_f$, the solution of \eqref{7:se}, we have
\begin{align*}
\sup_{f \in \mathcal{B}}|\Delta g_f (k)|=\frac{a_k}{\theta (k+1)a_{k+1}}\bar{F}(k+1)+\frac{1}{k}F(k-1),
\end{align*}
where $\Delta g_f(k)=g_f(k+1)-g_f(k)$.\\
Moreover, 
\begin{align}
\sup_{f \in \mathcal{B}}|\Delta g_f (k)|\le \frac{1}{k} \wedge \frac{a_k}{\theta (k+1)a_{k+1}},\label{7:bd}
\end{align}
where $x\wedge y$ denotes the minimum of $x$ and $y$.
\end{lemma}

\noindent
Now, it is not easy to use direct form of the Stein operator \eqref{7:sopsd} as $a_k$ is unknown and depends on $k$. So, we consider PSD family with Panjer's recursive relation (see Panjer and Wang \cite{PW} for details), denoted by $\mathcal{P}_1$, which is given by
\begin{equation}
(k+1)\frac{p_{k+1}}{p_k}=a+bk \implies \theta (k+1)\frac{a_{k+1}}{a_k}=a+bk,\quad\text{for~some}~a,b\in {\mathbb R}.\label{7:panj}
\end{equation}
Therefore, the stein operator \eqref{7:sopsd} can be written as
\begin{equation}
{\cal A}_Zg(k)=(a+bk)g(k+1)-k g(k), \quad k \in \mathbb{Z}_+.\label{7:pso}
\end{equation}
Also, the bound \eqref{7:bd} becomes
\begin{align}
\sup_{f \in \mathcal{B}}|\Delta g_f (k)|\le \frac{1}{k} \wedge \frac{1}{a+bk},\quad k\ge 1.\label{7:bdd}
\end{align}
Note that if $a,b\ge 0$ (PSD family satisfies Panjer recursive relation with $a,b\ge 0$, denoted by $\mathcal{P}_2$) then the bound \eqref{7:bdd} becomes uniform and is given by
\begin{align}
\sup_{f \in \mathcal{B}}|\Delta g_f (k)|\le 1 \wedge \frac{1}{a},\quad k\ge 1.\label{7:bddd}
\end{align}
Note that $\mathcal{P}_2\subset \mathcal{P}_1\subset \mathcal{P}$. Also, observe that $a=\lambda,b=0$ $\big(a_k=1/k!$, $\theta=\lambda$ and $\gamma(\theta)=e^{\theta}\big)$ and $a=nq,b=q$ $\big(a_k=\binom{n+k-1}{k}$, $\theta=q$ and $\gamma(\theta)=(1-\theta)^{-n}\big)$ for Poisson (with parameter $\lambda$) and negative binomial (with parameter $n$ and $p=1-q$) distributions, respectively, and hence the bounds (from \eqref{7:bddd}) are $1\wedge \frac{1}{\lambda}$ and $1\wedge \frac{1}{nq}$, respectively, which are well-known bounds for Poisson and negative binomial distributions. Many distributions satisfy the condition $a,b\ge 0$. However, if the condition is not satisfied, one can still use \eqref{7:bdd} to compute the uniform bound. For example, if $a_k=\binom{n}{k}$, $\theta=p/q$, and $\gamma(\theta)=(1+\theta)^n$, then $Z\sim \text{Bi}(n,p)$, and $\frac{\theta (k+1)a_{k+1}}{a_k}=\frac{p}{q}(n-k)$, and hence $a=np/q$ and $b=-p/q\le 0$. Therefore, the bound \eqref{7:bdd} is 
\begin{align}
\sup_{f \in \mathcal{B}}|\Delta g_f (k)|&\le \frac{1}{k} \wedge \frac{q}{p(n-k)}\nonumber\\
&=\left\{\begin{array}{ll}
\frac{1}{k} & \text{if $k \ge np$}\nonumber\\
\frac{q}{(n-k)p} & \text{if $k \le np$}
\end{array}\right.\nonumber\\
&\le\left\{\begin{array}{ll}
\frac{1}{np} & \text{if $k \ge np$}\nonumber\\
\frac{1}{np} & \text{if $k \le np$}
\end{array}\right.\\
&=\frac{1}{np},~\text{for all $k$},\label{7:bibd}
\end{align}
which leads to a uniform bound for binomial distribution. Note here that the Stein operator (from \eqref{7:pso}) is 
\begin{equation}
{\cal A}_Zg(k)=\frac{p}{q}(n-k)g(k+1)-kg(k).\label{7:oso}
\end{equation} 
But, the well-known Stein operator for the binomial distribution is 
\begin{equation}
{\cal A}_Zg(k)=p(n-k)g(k+1)-qkg(k),\label{7:sso}
\end{equation} 
which follows by multiplying $q$ in \eqref{7:oso}. Also, the uniform bound will be changed and is given by $1/npq$ (that is, divided by $q$), which is well-known bound with respect to the Stein operator \eqref{7:sso} (see Upadhye {\em et al.} \cite{UCV}). Hence, throughout this article, we use $\|\Delta g\|=\sup_k |\Delta g(k)|$ and the uniform bound for $\|\Delta g\|$ can be obtained from \eqref{7:bddd} or may be computed explicitly for some applications. 

\noindent
Next, let $\phi_Z(\cdot)$ be the probability generating function of $Z$. Then, using \eqref{7:panj}, it can be seen that
$$\phi_Z^\prime(t)=\frac{a \phi_Z(t)}{1-bt}.$$
Hence, mean and variance of the PSD are given by
\begin{equation}
{\mathbb E}(Z)=\frac{a}{1-b}\quad\text{and}\quad \mathrm{Var}(Z)=\frac{a}{(1-b)^2}.\label{7:mv}
\end{equation}
For more details, we refer the reader to Edwin \cite{ETK2014}, and Panjer and Wang \cite{PW}.

\section{Approximation Results}\label{7:AR}
Let $Y_1,Y_2,\ldots,Y_n$ be a sequence of $\mathbb{Z}_+$-valued $m$ dependent rvs and $S_n=\sum_{i=1}^{n}Y_i$, the sums of $m$-dependent rvs. Then, grouping the consecutive summations in the following form   
\begin{align*}
Y_i^*:=\sum_{j=(i-1)m+1}^{\min(im,n)}Y_j,\quad j=1,2,\ldots,\floor{n/m}+1,
\end{align*}
where $\floor{x}$ denotes the greatest integer function of $x$, $S_n=\sum_{j=1}^{\floor{n/m}+1}Y_j^*$ become the sums of $1$-dependent rvs. In this section, we derive an error bound for PSD approximation to the sums of $1$-dependent rvs in total variation distance and discuss some relevant remarks. Throughout this section, we assume $X_1,X_2,\dotsc,X_n$, $n\ge 1$, is a sequence of $1$-dependent rvs and
\begin{equation}
W_n=\sum_{i=1}^{n}X_i.\label{7:W}
\end{equation}
For any $\mathbb{Z}_+$-valued rv $Y$, let $D(Y):=2d_{TV}(Y,Y+1)$, where $d_{TV}(X,X^*)$ as defined in \eqref{dtv1}. Let 
\begin{align*}
N_{i,\ell}:=\{j:|j-i|\le \ell\}\cap \{1,2,\dotsc,n\}\quad \text{and} \quad X_{N_{i,\ell}}:=\sum_{j \in N_{i,\ell}}X_j,\quad\text{for}~\ell=1,2.
\end{align*}
Note that $X_{N_{i,2}}-X_{N_{i,1}}=X_{N_{i,2}-N_{i,1}}$. From \eqref{7:W}, it can be verified that ${\mathbb E}(W_n)=\sum_{i=1}^{n}{\mathbb E}(X_i)$ and
\begin{align}
\mathrm{Var}(W_n)&=\sum_{i=1}^{n}\sum_{|j-i|\le 1}\left[{\mathbb E}(X_i X_j)-{\mathbb E}(X_i){\mathbb E}(X_j)\right]=\sum_{i=1}^{n}\left[{\mathbb E}(X_i X_{N_{i,1}})-{\mathbb E}(X_i){\mathbb E}(X_{N_{i,1}})\right].\label{7:vari}
\end{align}

\noindent
Now, the following theorem gives the error bound for $Z$-approximation to $W_n$.

\begin{theorem}\label{7:th1}
Let $Z\in \mathcal{P}_1$ and $W_n$ be defined as in \eqref{7:W}. Assume that ${\mathbb E}(Z)={\mathbb E}(W_n)$, and $\tau=\mathrm{Var}(W_n)-\mathrm{Var}(Z)$. Then, for $n\ge 6$,
\begin{align}
d_{TV}(W_n, Z)&\le \|\Delta g\|\Bigg\{\frac{|1-b|}{2}\Bigg[\sum_{i=1}^{n}\mathbb{E}(X_i)\mathbb{E}[X_{N_{i,1}}(2 X_{N_{i,2}}-X_{N_{i,1}}-1)D(W_n|X_{N_{i,1}},X_{N_{i,2}})]\nonumber\\
& ~~~~~~~~~~~~~~~~~~~~+ \sum_{i=1}^{n}\mathbb{E}[X_iX_{N_{i,1}}(2 X_{N_{i,2}}-X_{N_{i,1}}-1)D(W_n|X_{N_{i,1}},X_{N_{i,2}})]\Bigg]\nonumber\\
& ~~~~~~~~~~~~~~~~~~~~+\sum_{i=1}^{n}\mathbb{E}[X_i (X_{N_{i,2}}-1)D(W_n|N_{i,2})]+|\tau(1-b)|\Bigg\}.\label{7:mainrs}
\end{align}
\end{theorem}

\begin{proof} Consider the Stein operator given in \eqref{7:pso} and taking expectation with respect to $W_n$, we have
\begin{align*}
{\mathbb E}\left[{\cal A}_Zg(W_n)\right]&=a{\mathbb E}[g(W_n+1)] +b{\mathbb E}[W_ng(W_n+1)]-{\mathbb E}[W_n g(W_n)]\\
&=a{\mathbb E}[g(W_n+1)] -(1-b){\mathbb E}[W_ng(W_n+1)]+{\mathbb E}[W_n \Delta g(W_n)]\\
&=(1-b)\left[\frac{a}{(1-b)}{\mathbb E}[g(W_n+1)] -{\mathbb E}[W_ng(W_n+1)]\right]+{\mathbb E}[W_n \Delta g(W_n)].
\end{align*}
Applying the first moment matching condition, ${\mathbb E}(Z)=a/(1-b)={\mathbb E}(W_n)$, we get
\begin{equation}
{\mathbb E}\left[{\cal A}_Zg(W_n)\right]=(1-b)\Big[{\mathbb E}(W_n){\mathbb E}[g(W_n+1)] -{\mathbb E}[W_ng(W_n+1)]\Big]+{\mathbb E}[W_n \Delta g(W_n)].\label{7:exp0}
\end{equation}
Let now
$$W_{i,n}:=W_n-X_{N_{i,1}}$$
so that $X_i$ and $W_{i,n}$ are independent. Consider the following expression from \eqref{7:exp0}  
\begin{align}
{\mathbb E}(W_n){\mathbb E}[g(W_n+1)] -{\mathbb E}[W_ng(W_n+1)&=\sum_{i=1}^{n}{\mathbb E}(X_i){\mathbb E}[g(W_n+1)]-\sum_{i=1}^{n}{\mathbb E}[X_ig(W_n+1)]\nonumber\\
&=\sum_{i=1}^{n}{\mathbb E}(X_i){\mathbb E}[g(W_n+1)]-\sum_{i=1}^{n}{\mathbb E}[X_ig(W_n+1)]\nonumber\\
&~~~~~-\sum_{i=1}^{n}{\mathbb E}[X_ig(W_{i,n}+1)]+\sum_{i=1}^{n}{\mathbb E}[X_ig(W_{i,n}+1)]\nonumber\\
&=\sum_{i=1}^{n}{\mathbb E}(X_i){\mathbb E}[g(W_n+1)-g(W_{i,n}+1)]\nonumber\\
&~~~~~-\sum_{i=1}^{n}{\mathbb E}[X_i(g(W_n+1)-g(W_{i,n}+1))].\label{7:www56}
\end{align}
It can be seen that
\begin{align}
g(W_n+1)-g(W_{i,n}+1)&=g(W_{i,n}+X_{N_{i,1}}+1)-g(W_{i,n}+1)\nonumber\\
&=g(W_{i,n}+X_{N_{i,1}}+1)-g(W_{i,n}+X_{N_{i,1}})\nonumber\\
&~~~~~+g(W_{i,n}+X_{N_{i,1}})-g(W_{i,n}+X_{N_{i,1}}-1)\nonumber\\
&~~~~~~~\vdots\nonumber\\
&~~~~~+g(W_{i,n}+2)-g(W_{i,n}+1)\nonumber\\
&=\sum_{j=1}^{X_{N_{i,1}}}\Delta g(W_{i,n}+j).\label{7:sss}
\end{align}
Using \eqref{7:sss} in \eqref{7:www56}, we get
\begin{align}
{\mathbb E}(W_n){\mathbb E}[g(W_n+1)] -{\mathbb E}[W_ng(W_n+1)]&=\sum_{i=1}^{n}{\mathbb E}(X_i){\mathbb E}\Bigg[\sum_{j=1}^{X_{N_{i,1}}}\Delta g(W_{i,n}+j)\Bigg]\nonumber\\
&~~~~~-\sum_{i=1}^{n}{\mathbb E}\Bigg[X_i\sum_{j=1}^{X_{N_{i,1}}}\Delta g(W_{i,n}+j)\Bigg].\label{7:exp11}
\end{align}
Substituting \eqref{7:exp11} in \eqref{7:exp0}, we have
\begin{align}
{\mathbb E}[{\cal A}_Zg(W_n)]&=(1-b)\left\{\sum_{i=1}^{n}{\mathbb E}(X_i){\mathbb E}\Bigg[\sum_{j=1}^{X_{N_{i,1}}}\Delta g(W_{i,n}+j)\Bigg]-\sum_{i=1}^{n}{\mathbb E}\Bigg[X_i\sum_{j=1}^{X_{N_{i,1}}}\Delta g(W_{i,n}+j)\Bigg]\right\}\nonumber\\
&~~~~~+\sum_{i=1}^{n}{\mathbb E}[X_i \Delta g(W_n)].\label{7:exp2}
\end{align}
Note that ${\mathbb E}(Z)=a/(1-b)={\mathbb E}(W_n)=\sum_{i=1}^{n}\mathbb{E}(X_i)$. Therefore, from \eqref{7:mv},
\begin{align*}
\mathrm{Var}(Z)=\frac{a}{(1-b)^2}=\frac{1}{(1-b)}\sum_{i=1}^{n}\mathbb{E}(X_i). 
\end{align*}
Hence,
\begin{align*}
\tau&=\mathrm{Var}(W_n)-\mathrm{Var}(Z)=\sum_{i=1}^{n}{\mathbb E}(X_i X_{N_{i,1}})-\sum_{i=1}^{n}{\mathbb E}(X_i){\mathbb E}(X_{N_{i,1}})-\frac{1}{(1-b)}\sum_{i=1}^{n}\mathbb{E}(X_i).
\end{align*}
This implies
\vspace{-0.33cm}
\begin{align}
(1-b)\left\{\sum_{i=1}^{n}{\mathbb E}(X_i X_{N_{i,1}})-\sum_{i=1}^{n}{\mathbb E}(X_i){\mathbb E}(X_{N_{i,1}})\right\}-\sum_{i=1}^{n}\mathbb{E}(X_i)-\tau(1-b)=0.\label{7:tau1}
\end{align}
Next, define
$$V_{i,n}:=W_n-X_{N_{i,2}}$$
so that $X_{N_{i,1}}$ and $V_{i,n}$ are independent, and $X_i$ and $V_{i,n}$ are independent. From \eqref{7:tau1}, we get
\begin{align*}
(1&-b)\left\{\sum_{i=1}^{n}{\mathbb E}(X_i X_{N_{i,1}}\Delta g(V_{i,n}))-\sum_{i=1}^{n}{\mathbb E}(X_i){\mathbb E}(X_{N_{i,1}}\Delta g(V_{i,n}))\right\}\\
&~~~~~~~~~~~~~~~~~~~~~~~~~~~~~~~~~~~~~~~~~~~~~~~~-\sum_{i=1}^{n}\mathbb{E}(X_i\Delta g(V_{i,n}))-\tau(1-b)\mathbb{E}(\Delta g(V_{i,n}))=0.
\end{align*}
This is equivalent to
\begin{align}
(1&-b)\left\{\sum_{i=1}^{n}{\mathbb E}\Bigg[X_i \sum_{j=1}^{X_{N_{i,1}}}\Delta g(V_{i,n})\Bigg]-\sum_{i=1}^{n}{\mathbb E}(X_i){\mathbb E}\Bigg[\sum_{j=1}^{X_{N_{i,1}}}\Delta g(V_{i,n})\Bigg]\right\}\nonumber\\
&~~~~~~~~~~~~~~~~~~~~~~~~~~~~~~~~~~~~~~~~~~~~~~~~-\sum_{i=1}^{n}\mathbb{E}(X_i\Delta g(V_{i,n}))-\tau(1-b)\mathbb{E}(\Delta g(V_{i,n}))=0.\label{7:tau2}
\end{align}
Using \eqref{7:tau2} in \eqref{7:exp2}, we get
\begin{align}
{\mathbb E}[{\cal A}_Zg(W_n)]&=(1-b)\left\{\sum_{i=1}^{n}{\mathbb E}(X_i){\mathbb E}\Bigg[\sum_{j=1}^{X_{N_{i,1}}}\Delta g(W_{i,n}+j)\Bigg]-\sum_{i=1}^{n}{\mathbb E}\Bigg[X_i\sum_{j=1}^{X_{N_{i,1}}}\Delta g(W_{i,n}+j)\Bigg]\right\}\nonumber\\
&~~~~~+(1-b)\left\{\sum_{i=1}^{n}{\mathbb E}\Bigg[X_i \sum_{j=1}^{X_{N_{i,1}}}\Delta g(V_{i,n})\Bigg]-\sum_{i=1}^{n}{\mathbb E}(X_i){\mathbb E}\Bigg[\sum_{j=1}^{X_{N_{i,1}}}\Delta g(V_{i,n})\Bigg]\right\}\nonumber\\
&~~~~~+\sum_{i=1}^{n}{\mathbb E}[X_i \Delta g(W_n)]-\sum_{i=1}^{n}\mathbb{E}(X_i\Delta g(V_{i,n}))-\tau(1-b)\mathbb{E}(\Delta g(V_{i,n}))\nonumber\\
&=(1-b)\left\{\sum_{i=1}^{n}{\mathbb E}(X_i){\mathbb E}\Bigg[\sum_{j=1}^{X_{N_{i,1}}}(\Delta g(W_{i,n}+j)-\Delta g(V_{i,n}))\Bigg]\right.\nonumber\\
&~~~~~\left.-\sum_{i=1}^{n}{\mathbb E}\Bigg[X_i\sum_{j=1}^{X_{N_{i,1}}}(\Delta g(W_{i,n}+j)-\Delta g(V_{i,n}))\Bigg]\right\}\nonumber\\
&~~~~~+\sum_{i=1}^{n}{\mathbb E}[X_i (\Delta g(W_n)-\Delta g(V_{i,n}))]-\tau(1-b)\mathbb{E}(\Delta g(V_{i,n})).\label{7:tau3}
\end{align}
Note also that
\begin{align}
\Delta g(W_{i,n}+j)-\Delta g(V_{i,n})&=\Delta g(V_{i,n}+X_{N_{i,2}-N_{i,1}}+j)-\Delta g(V_{i,n})\nonumber\\
&=\sum_{k=1}^{X_{N_{i,2}-N_{i,1}}+j-1}\Delta^2 g(V_{i,n}+k).\label{7:tau4}
\end{align}
and
\begin{align}
\Delta g(W_n)-\Delta g(V_{i,n})&=\Delta g(V_{i,n}+X_{N_{i,2}})-\Delta g(V_{i,n})\nonumber\\
&=\sum_{k=1}^{X_{N_{i,2}}-1}\Delta^2 g(V_{i,n}+k).\label{7:tau5}
\end{align}
Substituting \eqref{7:tau4} and \eqref{7:tau5} in \eqref{7:tau3}, we have
\begin{align}
{\mathbb E}[{\cal A}_Zg(W_n)]&=(1-b)\left\{\sum_{i=1}^{n}{\mathbb E}(X_i){\mathbb E}\Bigg[\sum_{j=1}^{X_{N_{i,1}}}\sum_{k=1}^{X_{N_{i,2}-N_{i,1}}+j-1}\Delta^2 g(V_{i,n}+k)\Bigg]\right.\nonumber\\
&~~~~~\left.-\sum_{i=1}^{n}{\mathbb E}\Bigg[X_i\sum_{j=1}^{X_{N_{i,1}}}\sum_{k=1}^{X_{N_{i,2}-N_{i,1}}+j-1}\Delta^2 g(V_{i,n}+k)\Bigg]\right\}\nonumber\\
&~~~~~+\sum_{i=1}^{n}{\mathbb E}\Bigg[X_i \sum_{j=1}^{X_{N_{i,2}}-1}\Delta^2 g(V_{i,n}+j)\Bigg]-\tau(1-b)\mathbb{E}(\Delta g(V_{i,n})).\label{7:tau6}
\end{align}
Consider first
\begin{align}
{\mathbb E}\Bigg[X_i \sum_{j=1}^{X_{N_{i,2}}-1}\Delta^2 g(V_{i,n}+j)\Bigg]&=\mathbb{E}\left\{{\mathbb E}\Bigg[X_i \sum_{j=1}^{X_{N_{i,2}}-1}\Delta^2 g(V_{i,n}+j)\Bigg|X_{N_{i,2}}\Bigg]\right\}\nonumber\\
&=\mathbb{E}\left\{{\mathbb E}\left(X_i|X_{N_{i,2}}\right){\mathbb E}\Bigg[\sum_{j=1}^{X_{N_{i,2}}-1}\Delta^2 g(W_n-X_{N_{i,2}}+j)\Bigg|X_{N_{i,2}}\Bigg]\right\},\label{7:tau8}
\end{align}
since $X_i$ and $V_{i,n}$ are independent given $X_{N_{i,2}}$. Observe that
\begin{align}
\left|{\mathbb E}\Bigg[\sum_{j=1}^{X_{N_{i,2}}-1}\Delta^2 g(W_n-X_{N_{i,2}}+j)\Bigg|X_{N_{i,2}}=n_{i,2}\Bigg]\right|\le \|\Delta g\||n_{i,2}-1|D(W_n|X_{N_{i,2}}=n_{i,2}).\label{7:tau7}
\end{align}
Using \eqref{7:tau7} in \eqref{7:tau8}, we have
\begin{align}
\left|{\mathbb E}\Bigg[X_i \sum_{j=1}^{X_{N_{i,2}}-1}\Delta^2 g(V_{i,n}+j)\Bigg]\right|&\le \|\Delta g\|\mathbb{E}[X_i|X_{N_{i,2}}-1|D(W_n|X_{N_{i,2}})]\nonumber\\
&=\|\Delta g\|\mathbb{E}[X_i(X_{N_{i,2}}-1)D(W_n|X_{N_{i,2}})],\label{7:third}
\end{align}
since $X_iX_{N_{i,2}}\ge X_i\implies X_i(X_{N_{i,2}}-1)\ge 0$.\\
Consider next the following expression from \eqref{7:tau6}
\begin{align}
{\mathbb E}\Bigg[\sum_{j=1}^{X_{N_{i,1}}}\sum_{k=1}^{X_{N_{i,2}-N_{i,1}}+j-1}\hspace{-0.38cm}\Delta^2 g(V_{i,n}+k)\Bigg]&=\mathbb{E}\left\{ {\mathbb E}\Bigg[\sum_{j=1}^{X_{N_{i,1}}}\sum_{k=1}^{X_{N_{i,2}}-X_{N_{i,1}}+j-1}\hspace{-0.61cm}\Delta^2 g(W_n\hspace{-0.08cm}-\hspace{-0.08cm}X_{N_{i,2}}\hspace{-0.08cm}+\hspace{-0.08cm}k)\Bigg| X_{N_{i,1}},X_{N_{i,2}}\Bigg]\right\}.\label{7:tau9}
\end{align}
Then
\begin{align}
&\left|{\mathbb E}\Bigg[\sum_{j=1}^{X_{N_{i,1}}}\sum_{k=1}^{X_{N_{i,2}-N_{i,1}}+j-1}\Delta^2 g(W_n-X_{N_{i,2}}+k)\Bigg| X_{N_{i,1}}=n_{i,1},X_{N_{i,2}}=n_{i,2}\Bigg] \right|\nonumber\\
&~~~~~~~~~~~~~~~~~~~~~~~~~~~~~~~~~~~~\le \frac{\|\Delta g\|}{2}n_{i,1}|2n_{i,2}-n_{i,1}-1|D(W_n|X_{N_{i,1}}=n_{i,1},X_{N_{i,2}}=n_{i,2}).\label{7:tau10}
\end{align}
Using \eqref{7:tau10} in \eqref{7:tau9}, we get
\begin{align}
\left|{\mathbb E}\Bigg[\sum_{j=1}^{X_{N_{i,1}}}\sum_{k=1}^{X_{N_{i,2}-N_{i,1}}+j-1}\Delta^2 g(V_{i,n}+k)\Bigg]\right|& \le \frac{\|\Delta g\|}{2}\mathbb{E}[X_{N_{i,1}}|2X_{N_{i,2}}-X_{N_{i,1}}-1|D(W_n|X_{N_{i,1}},X_{N_{i,2}})]\nonumber\\
&=\frac{\|\Delta g\|}{2}\mathbb{E}[X_{N_{i,1}}(2X_{N_{i,2}}-X_{N_{i,1}}-1)D(W_n|X_{N_{i,1}},X_{N_{i,2}})],\label{7:first}
\end{align}
since $X_{N_{i,2}}X_{N_{i,1}}-X_{N_{i,1}}^2\ge 0$ and $X_{N_{i,2}}X_{N_{i,1}}-X_{N_{i,1}}\ge 0$ which imply $X_{N_{i,1}}(2X_{N_{i,2}}-X_{N_{i,1}}-1)\ge 0$. Similarly,
\begin{align}
\left|{\mathbb E}\Bigg[X_i\sum_{j=1}^{X_{N_{i,1}}}\sum_{k=1}^{X_{N_{i,2}-N_{i,1}}+j-1}\hspace{-0.55cm}\Delta^2 g(V_{i,n}+k)\Bigg]\right| \le \frac{\|\Delta g\|}{2}\mathbb{E}[X_i X_{N_{i,1}}(2X_{N_{i,2}}-X_{N_{i,1}}-1)D(W_n|X_{N_{i,1}},X_{N_{i,2}})].\label{7:second}
\end{align}
Finally, using \eqref{7:third}, \eqref{7:first} and \eqref{7:second} in \eqref{7:tau6}, the proof follows.
\end{proof}

\begin{remarks}\label{7:re1}
\begin{itemize}
\item[(i)] For $n\ge 1$, we can use \eqref{7:exp0} to obtain the following crude upper bound for $d_{TV}(W_n,Z)$.
\begin{align}
d_{TV}(W_n,Z)&\le(2 |1-b| \|g\|+\|\Delta g\|)\sum_{i=1}^{n}\mathbb{E}(X_i).\label{7:mainrss}
\end{align}
Note however that for $n\ge 6$, the bound given in \eqref{7:mainrs} would better than the one given in \eqref{7:mainrss}.
\item[(ii)] Assume $D(W_n|X_{N_{i,2}})\le c_i(n)$ then $D(W_n|X_{N_{i,1}},X_{N_{i,2}})\le c_i(n)$. Therefore, the bound \eqref{7:mainrs} becomes
\begin{align}
d_{TV}(W_n,Z)&\le \|\Delta g\|\Bigg\{\sum_{i=1}^{n}c_i(n)\Bigg[\frac{|1-b|}{2}\Big[\mathbb{E}(X_i)\mathbb{E}[X_{N_{i,1}}(2X_{N_{i,2}}-X_{N_{i,1}}-1)]\nonumber\\
&~~~~~+\mathbb{E}[X_iX_{N_{i,1}}(2X_{N_{i,2}}-X_{N_{i,1}}-1)]\Big]+\mathbb{E}[X_i (X_{N_{i,2}}-1)]\Bigg]+|\tau(1-b)|\Bigg\}\nonumber\\
&=:d_1(n). \label{7:lll}
\end{align}
Furthermore, let us denote ${\cal L}(W_{i,n}^*)={\cal L}(W_n|X_{N_{i,2}})$ and $Z_e=\{X_{2m}~|~m \in \{1,\dotsc,\floor{n/2}\}\}=(X_2,X_4,\dotsc,X_{2\floor{n/2}})$. Then, ${\cal L}(W_{i,n}^*|Z_e=z_e)$ can be written as sum of independent rvs, say $X_{j}^{(z_e)}$, for $j=1,2,\dotsc,n_{z_e}$. Therefore, using $(5.11)$ of R\"{o}llin \cite{RO2008}, we have
\begin{align}
D(W_{i,n}^*)\le \mathbb{E}[\mathbb{E}[D(W_{i,n}^*)|Z_e]]\le \mathbb{E}\left[\frac{2}{V_{i,Z_e}^{1/2}}\right],\label{7:ve}
\end{align}
where
\begin{equation}
V_{i,z_e}=\sum_{j=1}^{n_{z_e}}\min\left\{\frac{1}{2},1-D\big(X_j^{(z_e)}\big)\right\}.\label{7:re11}
\end{equation} 
On the other hand, let
\begin{align}
m^*=\left\{\begin{array}{ll}
\floor{n/2}+1, & \text{if $n$ is odd}\\
n/2, & \text{if $n$ is even}
\end{array}\right.\quad \text{and}\quad Z_o&=\{X_{2m-1}~|~m \in \{1,\dotsc,m^*\}\}.\label{7:m*}
\end{align}
Then, applying the similar argument as above, we get 
\begin{align}
D(W_{i,n}^*)\le \mathbb{E}[\mathbb{E}[D(W_{i,n}^*)|Z_o]]\le \mathbb{E}\left[\frac{2}{V_{i,Z_o}^{1/2}}\right],\label{7:vo}
\end{align}  
where $V_{i,z_o}$ is defined in a similar way as $V_{i,z_e}$. Hence, from \eqref{7:ve} and \eqref{7:vo}, we have
\begin{align*}
D(W_{i,n}^*)\le \min\left\{\mathbb{E}\left[\frac{2}{V_{i,Z_o}^{1/2}}\right],\mathbb{E}\left[\frac{2}{V_{i,Z_e}^{1/2}}\right]\right\}=c_i(n).
\end{align*}
Note that $c_i(n)=O(n^{-1/2})$ in general. For more details, we refer the reader to Section $5.3$ and Section $5.4$ of R\"{o}llin \cite{RO2008}.
\item[(iii)] The bound given in Theorem \ref{7:th1} can also be used for the case of matching the first two moments (i.e., $\tau=0$), whenever that is possible with the approximating distribution.
\item[(iv)] From \eqref{7:exp2}, it can be easily verified that in the case of first moment matching, we have
\begin{align*}
d_{TV}(W_n,Z)\le \|\Delta g\|\left\{|1-b|\sum_{i=1}^{n}[\mathbb{E}(X_i)\mathbb{E}(X_{N_{i,1}})+\mathbb{E}(X_iX_{N_{i,1}})]+\sum_{i=1}^{n}\mathbb{E}(X_i)\right\}=:d_2(n).
\end{align*}
and then we have $d_{TV}(W_n,Z)\le \min\{d_1(n),d_2(n)\}$, where $d_{1}(n)$ is defined in \eqref{7:lll}.
\item[(v)] Observe that if $\tau=0$ then the bound given in \eqref{7:mainrs} is of optimal order $O(n^{-1/2})$ and is comparable with the existing bounds (Theorems $3.1$ $3.3$, and $3.4$ for Poisson, negative binomial, and binomial, respectively) given by \v{C}ekanavi\v{c}ius and Vellaisamy \cite{CV2015} with the relaxation of the conditions $(3.1)-(3.3)$.
\end{itemize}
\end{remarks}

\noindent
Next, we give a bound for any rv $X\in \mathcal{P}_2$ in the following corollary.
\begin{corollary}\label{7:coor}
Assume that the conditions of Theorem \ref{7:th1} hold. Then, for any $X\in \mathcal{P}_2$ and $n\ge 6$,
\begin{align}
d_{TV}(W_n, X)&\le \min \left\{1,\frac{1}{a}\right\}\Bigg\{\sum_{i=1}^{n}c_i(n)\Bigg[\frac{|1-b|}{2}\Big[\mathbb{E}(X_i)\mathbb{E}[X_{N_{i,1}}(2X_{N_{i,2}}-X_{N_{i,1}}-1)]\nonumber\\
&~~~~~+\mathbb{E}[X_iX_{N_{i,1}}(2X_{N_{i,2}}-X_{N_{i,1}}-1)]\Big]+\mathbb{E}[X_i (X_{N_{i,2}}-1)]\Bigg]+|\tau(1-b)|\Bigg\}.\label{7:aab}
\end{align}
\end{corollary}



\begin{example}
Assume that the conditions of Corollary \ref{7:coor} hold.  Moreover, let $Y\sim \text{Poi}(\lambda)$, the Poisson rv, so that $a=\lambda$ and $b=0$. Then, for $n\ge 6$,
\begin{align*}
d_{TV}(W_n, Y)&\le \min \left\{1,\frac{1}{\lambda}\right\}\Bigg\{\sum_{i=1}^{n}c_i(n)\Bigg[\frac{1}{2}\Big[\mathbb{E}(X_i)\mathbb{E}[X_{N_{i,1}}(2X_{N_{i,2}}-X_{N_{i,1}}-1)]\nonumber\\
&~~~~~+\mathbb{E}[X_iX_{N_{i,1}}(2X_{N_{i,2}}-X_{N_{i,1}}-1)]\Big]+\mathbb{E}[X_i (X_{N_{i,2}}-1)]\Bigg]+|\bar{\tau}_1|\Bigg\},
\end{align*}
where $\bar{\tau}_1=\mathrm{Var}(W_n)-\lambda$.
\end{example}

\begin{example}
Assume that the conditions of Corollary \ref{7:coor} hold.  Moreover, let $U\sim NB(\alpha,p)$, the negative binomial rv, so that $a=\alpha(1-p)$ and $b=1-p$. Then, for $n\ge 6$,
\begin{align*}
d_{TV}(W_n, U)&\le\min \left\{1,\frac{1}{\alpha(1-p)}\right\}\Bigg\{\sum_{i=1}^{n}c_i(n)\Bigg[\frac{p}{2}\Big[\mathbb{E}(X_i)\mathbb{E}[X_{N_{i,1}}(2X_{N_{i,2}}-X_{N_{i,1}}-1)]\nonumber\\
&~~~~~+\mathbb{E}[X_iX_{N_{i,1}}(2X_{N_{i,2}}-X_{N_{i,1}}-1)]\Big]+\mathbb{E}[X_i (X_{N_{i,2}}-1)]\Bigg]+|\bar{\tau}_2p|\Bigg\},
\end{align*}
where $\bar{\tau}_2=\mathrm{Var}(W_n)-\alpha (1-p)/p^2$.
\end{example}

\section{Applications to Runs}\label{7:APTK1K2R}
The distribution of runs and patterns has been applied successfully in many areas such as reliability theory, machine maintenance, quality control, and statistical testing, among many others. Also, it is not tractable if the underlying setup is arising from non-identical trials. So, the approximation of the runs has been studied by several researchers which includes, among others, Fu and Johnson \cite{FJ2009}, Godbole and Schaffner \cite{GS}, Kumar and Upadhye \cite{KU2019, KU2018}, Vellaisamy \cite{V}, and Wang and Xia \cite{WX}. In this section, we mainly focus on 2-runs and $(k_1,k_2)$-runs, however, the results can also be extended to other types of runs. 
\subsection{\texorpdfstring{$2$}--runs}\label{7:sub2r}
We consider here the setup similar to the one discussed in Chapter 5 of Balakrishnan and Koutras \cite[p.~166]{BK} for $2$-runs. Let $\eta_1,\eta_2,\dotsc,\eta_{n+1}$ be a sequence of independent Bernoulli trials with success probability $\mathbb{P}(\eta_i=1)=p_i=1-\mathbb{P}(\eta_i=0)$, for $i=1,2,\dotsc,n+1$. Assume that $p_i \le 1/2$, for all $i$, and
\begin{align}
R_n:=\sum_{i=1}^{n}X_i,\label{7:Wkk}
\end{align}
where $X_i=\eta_i \eta_{i+1}$, $1\le i \le n$, is a sequence of 1-dependent rvs. Observe that $R_n$ counts the number of overlapping success runs of length 2 in $n+1$ trials. It is easy to see that $\mathbb{E}X_i=\mathbb{P}(X_i=1)=p_i p_{i+1}:=a_1(p_i)$. Similarly, $\mathbb{E}(X_i X_{i+1})=p_i p_{i+1}p_{i+2}:=a_2(p_i)$ and $\mathbb{E}(X_i X_{i+1}X_{i+2})=p_i p_{i+1}p_{i+2}p_{i+3}:=a_3(p_i)$. Now, consider the first term in \eqref{7:mainrs}. Then
\begin{align}
\mathbb{E}(X_{N_{i,1}}(2 X_{N_{i,2}}-X_{N_{i,1}}-1))&=\mathbb{E}(X_{i-1}\hspace{-0.05cm}+\hspace{-0.05cm}X_i\hspace{-0.05cm}+\hspace{-0.05cm}X_{i+1})^2\hspace{-0.05cm}+\hspace{-0.05cm}\mathbb{E}[(X_{i-1}\hspace{-0.05cm}+\hspace{-0.05cm}X_i\hspace{-0.05cm}+\hspace{-0.05cm}X_{i+1})(2X_{i-2}\hspace{-0.05cm}+\hspace{-0.05cm}2X_{i+2}\hspace{-0.05cm}-\hspace{-0.05cm}1)]\nonumber\\
&=2\sum_{j=i-2}^{i+1}a_2(p_j)+2[a_1(p_{i-1})a_1(p_{i+1})+a_1(p_{i-2})(a_1(p_{i})+a_1(p_{i+1}))\nonumber\\
&~~~~~+a_1(p_{i+2})(a_1(p_{i-1})+a_1(p_{i}))]:=\bar{a}_1(p_i).\label{7:e44}
\end{align}
Similarly,
\begin{align}
\mathbb{E}(X_iX_{N{i,1}}(2 X_{N_{i,2}}-X_{N_{i,1}}-1))&=\mathbb{E}(X_i(X_{i-1}+X_i+X_{i+1})^2)\nonumber\\
&~~~~~+\mathbb{E}[X_i(X_{i-1}+X_i+X_{i+1})(2X_{i-2}+2X_{i+2}-1)]\nonumber\\
&=2a_1(p_i)(a_1(p_{i-2})+a_1(p_{i+2}))+2a_2(p_{i-1})(1+a_1(p_{i+2}))\nonumber\\
&~~~~~+2a_2(p_i)(1+a_1(p_{i-2}))+2\sum_{j=i-2}^{i}a_3(p_{j})=:\bar{a}_2(p_i).\label{7:e54}
\end{align}
and
\begin{align}
\mathbb{E}(X_{i}(X_{N_{i,2}}-1))&=\mathbb{E}\left(X_i\sum_{j=i-2}^{i+2}X_j\right)-\mathbb{E}(X_{i})\nonumber\\
&=a_1(p_i)\sum_{|j-i|=2}a_1(p_j)+\sum_{j=i-1}^{i}a_2(p_{j})=:\bar{a}_3(p_i). \label{7:e64}
\end{align}
Next, recall from Remarks \ref{7:re1} $(ii)$ with $W_n=R_n$ and $W_{i,n}^*=R_{i,n}^*$, ${\cal L}(R_{i,n}^*|Z_e=z_e)$ can be written as sum of independent rvs, say $X_j^{(z_e)}$, for $j\in\{1,2,\dotsc,n_{z_e}\}\cap\{\ell:|\ell-i|> 2\}=:{\cal C}_i$, for $i=1,2,\dotsc,n$. Note that $n_{z_e}=m^*$ defined in \eqref{7:m*} and $X_j^{(z_e)}=X_{2j-1}$ depends only on $X_{2j-2}$ $(2j\not\in \{2, i+4\})$ and $X_{2j}$ $(2j\neq i-2,~j \le \floor{n/2})$, $j\in {\cal C}_i$, for all values of $z_e$. So, for simplicity, let us write 
\begin{align*}
X_j^{(z_e)}=X_{2j-1}^{(x_{2j-2},x_{2j})}, \quad j\in{\cal C}_i,
\end{align*} 
where $x_{2j-2}$ and $x_{2j}$ are corresponding values of the rvs $X_{2j-2}$ and $X_{2j}$, respectively. Note that we use the same notation $D\big(X_{2j-1}^{(x_{2j-2},x_{2j})}\big)$, for $2j-1\in\{1,i-3,i+3,m^*\}$, while $X_{2j-1}$ depends either $X_{2j-2}$ or $X_{2j}$, not both. Therefore, from \eqref{7:re11}, we have
\begin{align*}
V_{i,z_e}=\displaystyle{\sum_{j\in{\cal C}_i}\min\left\{\frac{1}{2},1-D\left(X_{2j-1}^{(x_{2j-2},x_{2j})}\right)\right\}}.
\end{align*}
Note that $D\big(X_{2j-1}^{(1,1)}\big)=D\big(X_{2j-1}^{(1,0)}\big)=D\big(X_{2j-1}^{(0,1)}\big)=D\big(X_{2j-1}^{(0,0)}\big)=\frac{1}{2}$, for all $j\in \mathcal{C}_i$, except when $2j-1\in\{1,i-3,i+3,m^*\}$. For $2j-1\in\{1,i-3,i+3,m^*\}$, we have $D\big(X_{2j-1}^{(x_{2j-2},x_{2j})}\big)=p_{2j-1}$ or $p_{2j}\le 1/2\implies 1-D\big(X_{2j-1}^{(x_{2j-2},x_{2j})}\big)\ge 1/2$. Hence,  
\begin{align*}
V_{i,z_e}=\sum_{j\in{\cal C}_i}\min\left\{\frac{1}{2},1-D\left(X_{2j-1}^{(x_{2j-2},x_{2j})}\right)\right\}\ge \frac{1}{2}( m^*-3),\quad \text{for all values of $z_e$.}
\end{align*} 
Next, from \eqref{7:ve}, we have
\begin{align*}
D(R_{i,n}^*)\le \mathbb{E}\left[\frac{2}{V_{i,Z_e}^{1/2}}\right]\le 4 \left(m^*-3\right)^{-1/2},\quad \text{for all $i$}.
\end{align*}
Similarly,
\vspace{-0.24cm}
\begin{align*}
D(R_{i,n}^*)\le \mathbb{E}\left[\frac{2}{V_{i,Z_o}^{1/2}}\right]\le 4 \left(\floor{n/2}-3\right)^{-1/2},\quad \text{for all $i$}.
\end{align*}
Therefore,
\begin{align}
\bar{c}_i(n)=\bar{c}(n)&=4 \min\left\{\left(m^*-3\right)^{-1/2},\left(\floor{n/2}-3\right)^{-1/2}\right\}\le 4(m^*-3)^{-1/2},\quad \text{for all $i$}.\label{7:cine}
\end{align}
Hence, using \eqref{7:e44}, \eqref{7:e54}, \eqref{7:e64}, \eqref{7:cine}, and Theorem \ref{7:th1} and Remarks \ref{7:re1} $(ii)$, we obtain the following theorem. 
\begin{theorem}
Let $Z\in \mathcal{P}_1$ and $R_n$ be defined as in \eqref{7:Wkk}. Assume that ${\mathbb E}(Z)={\mathbb E}(R_n)$, and $\tau=\mathrm{Var}(R_n)-\mathrm{Var}(Z)$. Then, for $n \ge 8$ and $p_i\le 1/2$,
\begin{align}
d_{TV}(R_n,Z)&\le \|\Delta g\|\Bigg\{\bar{c}(n)\sum_{i=1}^{n}\Bigg[\frac{|1-b|}{2}\Big[a_1(p_i)\bar{a}_1(p_i)+\bar{a}_2(p_i)\Big]+\bar{a}_3(p_i)\Bigg]+|\tau(1-b)|\Bigg\}.\label{7:rruns}
\end{align}
\end{theorem}

\begin{remark}
Note that $\|\Delta g\|$ is of $O(n^{-1})$ in general, and hence, if $\mathrm{Var}(Z)=\mathrm{Var}(R_n)$ then the above bound become of order $O(n^{-1/2})$ and is comparable with the bounds given by Barbour and Xia \cite{BX}, Brown and Xia \cite{BX2001}, Daly {et al.} \cite{daly2012}, and Wang and Xia \cite{WX}. In fact, if $p_i=p$, for all $1 \le i\le n+1$, then $a(p)=p^2$, $\bar{a}_1(p)=8p^3+10p^4$, $\bar{a}_2(p)=4p^3+10p^4+4p^5$, and $\bar{a}_3(p)=2p^3+4p^4$, for all $1\le i \le n+1$, and hence, we have from \eqref{7:rruns},
\begin{align}
d_{TV}(R_n,Z)&\le n\|\Delta g\|\bar{c}(n)\left[\frac{|1-b|}{2}\left[4p^3+10p^4+12p^5+10p^6\right]+2(p^3+p^4)\right].\label{7:rruns1}
\end{align}
Now, let $Z\sim$NB$(\alpha,\bar{p})$, the negative binomial distribution with parameter $\alpha$ and $\bar{p}$. Then, $b=1-\bar{p}$ with $\bar{p}=1/(1+2p-3p^3)$ and $\|\Delta g\|\le \frac{1}{\alpha(1-\bar{p})}=\frac{1+2p-3p^2}{np^2}$, where $\alpha$ and $\bar{p}$ are obtained from first two moments matching condition, and hence, for $n \ge 8$ and $p\le 1/2$, we get
\begin{align}
d_{TV}(R_n,\text{NB}(\alpha,\bar{p}))\le \frac{4p}{(m^*-3)^{1/2}}[4+11p+4p^2-p^3].\label{7:iden}
\end{align} 
Also, from Theorem $4.2$ of Brown and Xia \cite{BX2001}, for $n\ge 2$ and $p<2/3$, we have
\begin{align}
d_{TV}(R_n,\text{NB}(\alpha,\bar{p}))\le \frac{32.2 p}{\sqrt{(n-1)(1-p)^3}}.\label{7:bx1}
\end{align}
For $n \ge 8$, we compare our bound with the one given in \eqref{7:bx1}, due to Brown and Xia \cite{BX2001}. Some numerical computations are given in the following table.

\begin{table}[H]
  \centering
  \caption{Comparison of bounds.}
  \begin{tabular}{cccccccccccc}
 \toprule
$n$ & $p$ & From \eqref{7:iden} & From \eqref{7:bx1} & $n$ & From \eqref{7:iden} & From \eqref{7:bx1} & $n$ & From \eqref{7:iden} & From \eqref{7:bx1}\\
\midrule
\multirow{3}{*}{20} & 0.05 & 0.344694 & 0.398900 & \multirow{3}{*}{25} & 0.303992 & 0.354924& \multirow{3}{*}{30} & 0.263265 & 0.322880\\
 & 0.07 & 0.506847 & 0.576571 & & 0.446997 & 0.513008&& 0.387111 & 0.466692\\
 & 0.09 & 0.683285 & 0.765878 & & 0.602601 & 0.681445&& 0.521867 & 0.619922\\
\midrule
\multirow{3}{*}{35} & 0.11 & 0.618205 & 0.723476 & \multirow{3}{*}{40} &  0.561012 & 0.675509& \multirow{3}{*}{50} &  0.493157 & 0.602650\\
 & 0.13 & 0.763728 & 0.884669 &  & 0.693072 & 0.826015 && 0.609244 & 0.736923\\
 & 0.15 & 0.919907 & 1.057010 &  & 0.834802 & 0.986930 && 0.733832 & 0.880482 \\
\bottomrule
\end{tabular}
\end{table}

\noindent
It is clear from the above table that our bound given in \eqref{7:iden} is better than the bound given in \eqref{7:bx1}, which is due to Brown and Xia \citep{BX2001}.
\end{remark}

\subsection{\texorpdfstring{$(k_1,k_2)$}--runs}\label{7:k1k2}
In this subsection, we consider the setup similar to Huang and Tsai \cite{HT} and Vellaisamy \cite{V}. Let $I_1,I_2,\dotsc$ be a sequence of independent Bernoulli trials. Here, we consider $I_1,I_2,\ldots,I_{(n+1)(k_1+k_2-1)}$ with success probability $\mathbb{P}(I_i=1)=p_i=1-\mathbb{P}(I_i=0)$, for $i=1,2,\dotsc,(n+1)(k_1+k_2-1)$. Define $m:=k_1+k_2-1$ and 
\begin{align}
Y_j:=(1-I_j)\dotsc (1-I_{j+k_1-1})I_{j+k_1}\dotsc I_{j+k_1+k_2-1},\quad j=1,2,\dotsc,nm.\label{7:k1}
\end{align}
Note that $Y_1,Y_2,\dotsc,Y_{nm}$ is a sequence of $m$-dependent rvs. Now, let us also define
\begin{align*}
X_i:=\sum_{j=(i-1)m+1}^{im}Y_j,\quad \text{for $i=1,2,\dotsc,n$}.
\end{align*} 
Then $X_1,X_2,\dotsc,X_n$ become a sequence of $1$-dependent rvs, that is, we reduced $m$-dependent to $1$-dependent sequence of rvs. From \eqref{7:k1}, it is clear that $Y_i$, for $i=1,2,\dotsc,nm$, are Bernoulli rvs and if $Y_i=1$ then $Y_j=0$, for all $j$ such that $|j-i|\le m$ and $j \neq i$. Therefore, $X_i$, $i=1,2,\dotsc,n$, are also Bernoulli rvs. Next, let
\begin{align}
R_n^\prime=\sum_{i=1}^{nm}Y_i=\sum_{i=1}^{n}X_i,\label{7:Wk}
\end{align} 
the sum of the corresponding $1$-dependent rvs ($X_i$'s). The distribution of $R_n^\prime$ is called the distribution of $(k_1,k_2)$-runs or modified distribution of order $k$ or distribution of order $(k_1,k_2)$. For more details. we refer the reader to Balakrishnan and Koutras \cite{BK}, Huang and Tsai \cite{HT},  Upadhye and Kumar \cite{KU2018}, Vellaisamy \cite{V} and reference therein.\\
Next, note that 
\begin{align*}
\mathbb{E}(Y_j)=\mathbb{P}(Y_j=1)=(1-p_j)\dotsc(1-p_{j+k_1-1})p_{j+k_1}\dotsc p_{j+k_1+k_2-1}=:a(p_j),~~~ \text{for } j=1,2,\dotsc,nm,
\end{align*} and hence 
\begin{align}
\mathbb{E}X_i=\sum_{j=(i-1)m+1}^{im}\mathbb{E}Y_j=\sum_{j=(i-1)m+1}^{im}a(p_j)=:a^*(p_i),\quad \text{for $i=1,2,\dotsc,n$}.\label{7:e1}
\end{align}
Also,
\begin{align}
\mathbb{E}(X_i X_{i+1})&=\sum_{\ell_1=(i-1)m+1}^{im-1}a(p_{\ell_1})\sum_{\ell_2=\ell_1+m+1}^{(i+1)m}a(p_{\ell_2})+\sum_{\ell_1=im+2}^{(i+1)m}a(p_{\ell_1})\sum_{\ell_2=(i-1)m+1}^{\ell_1-m-1}a(p_{\ell_2})=:a^*(p_ip_{i+1}).\label{7:e2}
\end{align}
and
\begin{align}
\mathbb{E}(X_{i}X_{i+1}X_{i+2})&=\sum_{\ell_1=(i-1)m+1}^{im-1}a(p_{\ell_1})\sum_{\ell_2=\ell_1+m+1}^{(i+1)m-1}a(p_{\ell_2})\sum_{\ell_3=\ell_2+m+1}^{(i+2)m}a(p_{\ell_3})\nonumber\\
&~~~~~+\sum_{\ell_1=im+2}^{(i+1)m-1}a(p_{\ell_1})\sum_{\ell_2=(i-1)m+1}^{\ell_1-m-1}a(p_{\ell_2})\sum_{\ell_3=\ell_1+m+1}^{(i+2)m}a(p_{\ell_3})\nonumber\\
&~~~~~+\sum_{\ell_1=(i+1)m+3}^{(i+2)m}a(p_{\ell_1})\sum_{\ell_2=im+2}^{\ell_1-m-1}a(p_{\ell_2})\sum_{\ell_3=(i-1)m+1}^{\ell_2-m-1}a(p_{\ell_3})\nonumber\\
&=:a^*(p_{i}p_{i+1}p_{i+2}).\label{7:e3}
\end{align}
Using the steps similar to \eqref{7:e44}-\eqref{7:e64} with \eqref{7:e1}, \eqref{7:e2} and \eqref{7:e3}, we have
\begin{align}
\mathbb{E}(X_{N_{i,1}}(2 X_{N_{i,2}}-X_{N_{i,1}}-1))&\le 2\sum_{j=i-2}^{i+1}a^*(p_jp_{j+1})+2[a^*(p_{i-1})a^*(p_{i+1})\nonumber\\
&~~~~~+a^*(p_{i-2})(a^*(p_{i})+a^*(p_{i+1}))+a^*(p_{i+2})(a^*(p_{i-1})+a^*(p_{i}))]\nonumber\\
&:=a_1^*(p_i),\label{7:e4}\\
\mathbb{E}(X_iX_{N_{i,1}}(2 X_{N_{i,2}}-X_{N_{i,1}}-1))&\le 2a^*(p_i)(a^*(p_{i-2})+a^*(p_{i+2}))+2a^*(p_{i-1}p_i)(1+a^*(p_{i+2}))\nonumber\\
&~~~~~+2a^*(p_ip_{i+1})(1+a^*(p_{i-2}))+2\sum_{j=i-2}^{i}a^*(p_{j}p_{j+1}p_{j+2})\nonumber\\
&=:a_2^*(p_i).\label{7:e5}
\end{align}
and
\begin{align}
\mathbb{E}(X_{i}(X_{N_{i,2}}-1))&\le a^*(p_i)\sum_{|j-i|=2}a^*(p_j)+\sum_{j=i-1}^{i}a^*(p_{j}p_{j+1})=:a_3^*(p_i). \label{7:e6}
\end{align}
Next, from Subsection \ref{7:sub2r}, following the discussion about Remarks \ref{7:re1} $(ii)$, we have
$$V_{i,z_e}=\displaystyle{\sum_{j\in{\cal C}_i}\min\left\{\frac{1}{2},1-D\left(X_{2j-1}^{(x_{2j-2},x_{2j})}\right)\right\}}.$$
Note that
\begin{align}
D\left(X_{2j-1}^{(x_{2j-2},x_{2j})}\right)&=\frac{1}{2}\left[\mathbb{P}\left(X_{2j-1}^{(x_{2j-2},x_{2j})}=0\right)+\left|\mathbb{P}\left(X_{2j-1}^{(x_{2j-2},x_{2j})}=0\right)-\mathbb{P}\left(X_{2j-1}^{(x_{2j-2},x_{2j})}=1\right)\right|\right]\nonumber\\
&\le \frac{1}{2}\left[1+\mathbb{P}\left(X_{2j-1}^{(x_{2j-2},x_{2j})}=0\right)\right]\nonumber\\
& \le \frac{1}{2}[1+\bar{a}(p_{2j-1})],\label{7:ww}
\end{align}
where 
\begin{align}
\bar{a}(p_{2j-1})=\displaystyle{\max_{0\le x_{2j-2},~x_{2j}\le 1}\mathbb{P}\left(X_{2j-1}^{(x_{2j-2},x_{2j})}=0\right)}.\label{7:apj}
\end{align} 
Next, using \eqref{7:ww}, we have
\begin{align*}
 \frac{1}{V_{i,z_e}^{1/2}}\le \left( \frac{1}{2}\min\left\{1,\sum_{j\in{\cal C}_i}(1-\bar{a}(p_{2j-1}))\right\}\right)^{-1/2},~\text{for all $z_e$}
\end{align*}
Therefore, from \eqref{7:ve}, we have
\begin{align*}
D(R_{i,n}^{\prime^*})\le \mathbb{E}\left[\frac{2}{V_{i,Z_e}^{1/2}}\right]\le 2 \left( \frac{1}{2}\min\left\{1,\sum_{j\in{\cal C}_i}(1-\bar{a}(p_{2j-1}))\right\}\right)^{-1/2}=:V_{i,e}^*.
\end{align*}
Similarly,
\begin{align*}
D(R_{i,n}^{\prime^*})\le \mathbb{E}\left[\frac{2}{V_{i,Z_o}^{1/2}}\right]\le 2 \left( \frac{1}{2}\min\left\{1,\sum_{j\in{\cal D}_i}(1-\bar{a}(p_{2j-1}))\right\}\right)^{-1/2}=:V_{i,o}^*,
\end{align*}
where ${\cal D}_i=\{1,2,\dotsc,\floor{n/2}\}\cap\{\ell:|\ell-i|>2\}$.
Therefore,
\begin{align}
c_i^*(n)=\min\{V_{i,e}^*,V_{i,o}^*\}.\label{7:cin}
\end{align}
Using \eqref{7:e1}, \eqref{7:e4}-\eqref{7:e6}, \eqref{7:cin}, Theorem \ref{7:th1} and Remarks \ref{7:re1} $(ii)$, the following result is established. 
\begin{theorem}
Let $Z\in \mathcal{P}_1$ and $R_n^\prime$ be defined as in \eqref{7:Wk}. Assume that ${\mathbb E}(Z)={\mathbb E}(R_n^\prime)$, and $\tau=\mathrm{Var}(R_n^\prime)-\mathrm{Var}(Z)$. Then, for $n \ge 3m$, $a(p_{2j-1})\le 1/3$ defined in \eqref{7:apj},
\begin{align*}
d_{TV}(R_n^\prime,Z)&\le \|\Delta g\|\Bigg\{\sum_{i=1}^{n}c_i^*(n)\Bigg[\frac{|1-b|}{2}\Big[a^*(p_i)a_1^*(p_i)+a_2^*(p_i)\Big]+a_3^*(p_i)\Bigg]+|\tau(1-b)|\Bigg\}.\end{align*}
\end{theorem}

\begin{remark}
Note that the above bound is comparable with the existing bounds given by Upadhye and Kumar \cite{KU2018} and order improvement over the bounds given by Barbour {et al.} \cite{BHJ}, Godbole \cite{G1993}, Godbole and Schaffner \cite{GS} (with $k_1=1$), and Vellaisamy \cite{V}. Also, note that we have used a slightly different form of $(k_1,k_2)$-runs, that is, we use $I_1,I_2,\dotsc,I_{(n+1)(k_1+k_2-1)}$ instead of $I_1,I_2,\dotsc,I_n$, so that $X_1,X_2,\dotsc,X_n$ become a sequence of 1-dependent rvs and we can directly apply our result. However, we can also use some other forms and derive the corresponding results.
\end{remark}

\singlespacing
\bibliographystyle{PV}
\bibliography{PSDFMD} 
\end{document}